\documentclass{amsart}

\usepackage{amsmath, amsthm, amssymb}

\newtheorem{theorem}{Theorem}
\newtheorem{prop}{Proposition}
\newtheorem{cor}{Corollary}
\newtheorem{lemma}{Lemma}

\newtheorem{refthm}{Theorem}

\newcommand{\D}{\mathbb{D}}
\newcommand{\C}{\mathbb{C}}
\newcommand{\conj}[1]{\overline{#1}}
\newcommand{\close}[1]{\overline{#1}}

\DeclareMathOperator{\sgn}{sgn}
\DeclareMathOperator{\Rp}{Re}

\newcommand{\wea}{\omega}  %
\newcommand{\web}{\nu}

\begin{document}

\title[Extremal Function Regularity in Weighted Bergman and Fock Spaces]
{Regularity of Extremal Functions in Weighted Bergman and 
Fock Type Spaces}


\author{Timothy Ferguson}

\date{October 30, 2014}

\begin{abstract}
We discuss the regularity of 
extremal functions in certain weighted Bergman and Fock type spaces. 
Given an appropriate analytic function $k$, the 
corresponding extremal function is the function with unit norm 
maximizing $\Rp \int_\Omega f(z) \conj{k(z)}\, \nu(z) \, dA(z)$
over all functions $f$ of unit norm, 
where $\nu$ is the weight function and 
 $\Omega$ is the domain of the functions in the space. 
We consider the case where $\nu(z)$ is a decreasing radial function 
satisfying some additional assumptions, and where 
$\Omega$ is either a disc centered at the origin or the entire 
complex plane. 
We show that if $k$ grows slowly in a certain sense, then 
$f$ must grow slowly in a related sense.  
We also discuss a relation between the integrability and growth 
of certain log-convex functions, and apply the result to 
obtain information about the growth of integral means 
of extremal functions in Fock type spaces.
\end{abstract}

\maketitle

This article deals with the regularity of solutions to 
extremal problems in certain weighted Bergman spaces 
in discs, as well as in Fock spaces.  Our results also apply
to other spaces of entire functions that are similar to 
Fock spaces but that are defined using a measure other than 
$e^{-\alpha |z|^2} \, dA$.  (For information on Bergman spaces, see the 
books \cite{D_Ap} or \cite{Zhu_Ap}.  For information on Fock 
spaces, see for example \cite{Zhu_Fock}.)

For Hardy spaces, which have many similarities with Bergman
spaces but are often simpler to study, 
extremal problems have been extensively investigated 
(see \cite{D_Hp} for references). 
Extremal problems in 
Bergman spaces are an area of active research.  
For example, see \cite{Dragan}, \cite{Khavinson_nonvanishing}, 
\cite{Dragan_point_eval}, 
and \cite{MacGregor_Stessin}.
One important application of extremal problems in Bergman spaces is 
to the study of canonical divisors, 
which appear as solutions to certain extremal problems 
and play a role in Bergman spaces similar to Blaschke products in 
Hardy spaces
(see \cite {Hedenmalm_canonical_A2}, \cite{Hansbo}, \cite{DKS}, 
\cite{DKSS_Pac}, and \cite{DKSS_Mich}). 
Extremal functions in Fock spaces 
have been studied in \cite{Benetau_Fock_Extremal}. 

Several results about regularity of solutions to extremal problems in 
Bergman spaces are known, although there are many open questions.  
In \cite{Ryabykh}, Ryabykh obtained an important result on the subject 
(see \cite{tjf1} for a simplified proof).  The articles 
\cite{tjf2}, \cite{Khavinson_McCarthy_Shapiro}, 
\cite{Sundberg}, and
\cite{Khavinson_Stessin} also deal with regularity of solutions 
to extremal problems in Bergman spaces.

We now discuss the subject of this paper in more detail. 
Let $0 < R \le \infty$ and let $\D_R$ be the open disc of radius $R$ centered 
at the origin (if $R=\infty$, then $\D_R$ is the entire complex plane). 
Let $\nu$ be a non-negative measurable function on $\D_R$ 
that is different from zero on a set of positive measure 
and let $A^p_R(\nu)$ be the space 
of all functions analytic in $\D_R$ such that 
\[
\|f\| = \left\{ \int_{\D_R} |f(z)|^p \nu(z) \, dA(z) \right\}^{1/p} < \infty.
\]
Throughout the paper we make the assumption that $1<p<\infty$ unless
otherwise noted. 
For certain functions $\nu$, the space $A^p_R(\nu)$ is a Banach
space with norm $\|\cdot\|$.  When $R<\infty$, the space is known 
as a weighted Bergman space, whereas when $R=\infty$, the space 
will be called a Fock type space. The standard Fock spaces correspond to the 
case where $R = \infty$ and $\nu = e^{-\alpha |z|^2}$, where 
$\alpha > 0$. 

Let $\nu$ be a function such that $A^p_R(\nu)$ is a Banach space, 
and suppose that $k \in A^{p'}_R(\nu)$, where $1/p + 1/p' = 1$. Then 
\[
f \mapsto \int_{\D_R} f(z) \conj{k(z)} \, \nu(z)\, dA(z)
\]
defines a linear functional $\Phi_k$ on $A^p_R(\nu)$, 
with norm at most $\|k\|_{A^{p'}_R(\nu)}$.  
We let $\|k\|^*$ denote the norm of $\Phi_k$. Thus, 
$\|k\|^* \le \|k\|_{A^{p'}_R(\nu)}$ for all functions $k \in A^{p'}_R(\nu)$.

We seek a function $f \in A_R^p(\nu)$ such that 
\begin{equation}\label{eq:extprob}
\|f\|_{A_R^p(\nu)} = 1 \quad \text{and} \quad
  \Rp \Phi_k(f) = \sup_{\|g\|_{A_R^p(\nu)} = 1} \Rp \int_{\D_R} f(z) \conj{k(z)} 
\, \nu(z) \, dA(z) .
 \end{equation}
We say that $k$ is the integral kernel for the extremal problem, and that 
$f$ is the corresponding extremal function. 
Because the space 
$L^p(\nu)$ is uniformly convex, there always exists a unique 
solution to this extremal problem (see \cite{tjf1}, Theorem 1.4).
In the case where $\nu=1$ and $R < \infty$, it is know that 
if $k$ has some suitable additional regularity beyond being 
in the space $A^{p'}_R(\nu)$, then $f$ will also have some additional 
regularity.  For example, see 
\cite{Ryabykh}, \cite{tjf1}, \cite{tjf2}, and \cite{Khavinson_Stessin}.
In what follows, we generalize the results of Ryabykh in \cite{Ryabykh} 
to certain non-constant measures $\nu$, and we obtain results for both the 
cases $R<\infty$ and $R=\infty$.

The outline of this article is as follows.  In Section \ref{sec:prelim}, 
we discuss some preliminary results.  In Section \ref{sec:regdisc}, we 
discus regularity for extremal functions in weighted Bergman spaces, and 
in Section \ref{sec:regplane}, we discuss regularity for 
weighted Fock type 
spaces.  In Section \ref{sec:logconvex}, we give results which throw 
further light on some of the quantities appearing in the statement of 
the main theorem of Section \ref{sec:regplane}. 
To do this, we find a relation between the integrability and growth 
of certain log-convex functions, and apply the result to 
obtain information about the growth of integral means 
of extremal functions in Fock type spaces.
In Section \ref{sec:polydense}, 
we discuss the density of polynomials in various weighted 
Bergman and Fock type spaces and present various auxiliary results 
which are needed for the main results of the paper. 

\section{Some Preliminary Results}\label{sec:prelim}

Let $0 < R \le \infty$ and 
let $\D_R$ denote the open disc centered at the origin with radius $R$
(where $\D_\infty = \C$).  Let $dA$ represent area measure. 

For $\nu(z)$ a non-negative measurable function defined on 
$\D_R$ that is not identically zero (in the almost everywhere sense), 
we let $A^p_R(\nu)$ be the space of all
functions analytic in $\D_R$ that are also in $L^p(\nu \, dA)$.  We take the 
norm of $A^p_R(\nu)$ to be the same as the norm of $L^p(\nu \, dA)$. 
Note that while $A^p_R(\nu)$ is a subspace of $L^p(\nu \, dA)$, it is not 
necessarily a closed subspace. However, for all the measures we deal with, 
$A^p_R(\nu)$ will be a closed subspace of $L^p(\nu \, dA)$. 

Many of our results focus on the case where $\nu(z) = \wea(|z|^2)$, where
$\wea$ is a positive, decreasing, and non-constant  
function on $[0,R^2)$ 
that is analytic in some complex neighborhood of $[0,R^2)$. 
The space $A^p_R(\wea(|z|^2))$ has norm defined by  
\[
\|f\|_{A^p_R(\wea(|z|^2))} = 
      \left( \int_{\D_R} |f(z)|^p \wea(|z|^2)\, dA(z) \right)^{1/p}, 
\]
where $dA$ represents area measure. 
We note that the space in question is indeed a Banach space, by 
Proposition \ref{prop:apbanach}.

Next, we recall the Cauchy-Green theorem, which we state for 
convenience since we will be using it several times.  
\begin{refthm}
Let $\Omega$ be a 
$C^1$ domain in $\mathbb{C}$ and let $f \in C^1(\close{\Omega}).$ 
Then 
\[
\begin{split}
\frac{1}{2i} \int_{\partial \Omega} f(z)\, dz &= 
\int_{\Omega} \frac{\partial}{\partial \overline{z}} f(z) \, dA 
\text{\qquad and} \\
\frac{i}{2} \int_{\partial \Omega} f(z)\, d\conj{z} &= 
\int_{\Omega} \frac{\partial}{\partial z} f(z) \, dA. \\
\end{split}
\]
\end{refthm}

We will also need the following theorem, which gives a characterization 
of extremal functions. It can be found in \cite{Shapiro_Approx}, p.~55.
\begin{refthm}\label{thm:exint}
Let $\sigma$ be a measure, 
let $1 < p < \infty$, let  $X$ be a closed 
subspace of $L^p(\sigma)$, and let $\phi \in X^*$, 
the dual space of $X$. 
Assume that $\phi$ is not identically $0$. 
A function $F \in X$ with $\|F\| = 1$ satisfies 
$$\Rp \phi(F) = \sup_{g \in X, \|g\|=1} \Rp \phi(g) = \| \phi\|_{X^*}$$
if and only if $\phi(F) > 0$ and 
$$\int h |F|^{p-1} \conj{\sgn F}  \, d\sigma = 0$$ for all 
$h \in X$ with $\phi(h) = 0.$  
If $F$ satisfies the above conditions, then 
$$\int h |F|^{p-1} \conj{\sgn F}\,  d\sigma 
= \frac{\phi(h)}{\| \phi \|_{X^*}}$$
for all $h\in X.$
\end{refthm}

Lastly, we note that 
$L^p$ spaces are uniformly convex for $1<p<\infty$
(see \cite{Clarkson} for a definition of uniform convexity and a proof of 
this result), and thus 
the $A^p_R$ spaces under consideration are uniformly convex,
since any closed subspace of a uniformly convex space is uniformly convex.
Using Theorem 3.1 in \cite{tjf1} and the fact that 
$\|k\|^* \le \|k\|_{A^{p'}_R(\nu)}$ for all functions $k\in A^{p'}_R(\nu)$, 
we have the 
following theorem. 
\begin{refthm}\label{thm:extkernelcont}
Suppose that $A^p_R(\nu)$ is a Banach space and that $k$ is a 
non-zero function in $A^{p'}_R(\nu)$. 
Then there is a unique solution to the extremal problem 
\eqref{eq:extprob} with integral kernel $k$.  Let $k_n$ be 
a sequence of functions approaching $k$ in the 
$A^{p'}_R(\nu)$ norm, and let $f_n$ be the extremal 
functions corresponding to $k_n$, and let $f$ be the extremal function 
corresponding to $k$.  Then $f_n \rightarrow f$ in the $A^p_R(\nu)$ norm. 
\end{refthm}

By Theorem 4.1 in \cite{tjf1}, we have the following result. 
When we apply it, we will let $X_n$ be the space of polynomials 
of degree at most $n$.  
\begin{refthm}\label{thm:extspacecont}
Suppose that $A^p_R(\nu)$ is a Banach space, 
that $f$ is the solution to the extremal problem 
\eqref{eq:extprob} with integral kernel $k$, and that 
$X_1 \subset X_2 \subset \cdots $ are closed subspaces of 
$A^p_R(\nu)$ such that 
$\close{ \cup_{n=1}^\infty X_n} = A^p_R(\nu)$. Let 
$f_n$ be the solution to the extremal problem 
\eqref{eq:extprob} posed over the space $X_n$ instead of the 
space $A^p_R(\nu)$.  Then $f_n$ exists and is unique, and 
$f_n \rightarrow f$ in the $A^p_R(\nu)$ norm as $n \rightarrow \infty$. 
Also, $\| \Phi_{k}|_{X_n} \| \rightarrow \|\Phi_k\|$ as $n \rightarrow \infty$.
\end{refthm}

\section{Regularity of Extremal Functions in Weighted Bergman Spaces}\label{sec:regdisc}

Let $R < \infty$. 
We suppose that $\wea$ is analytic in a neighborhood 
of $[0,R^2)$, and that $\wea$ is positive and decreasing
on $[0,R^2)$. 
This implies that $\wea$ has a limit from the left at $R^2$, so 
we may assume without loss of generality that it is continuous from 
the left at $R^2$. 
By Proposition \ref{prop:appolydensedisc}, the polynomials are 
dense in $A^p_R(\wea(|z|^2))$. 
Now suppose $f$ is analytic in the disk $\D_R$ and 
is in $C^1(\close{\D_R})$. 
Consider the integral
\[
\frac{R^2}{2} \int_{0}^{2\pi} |f(Re^{i\theta})|^p \, \wea(R^2) \, d\theta.
\]
We let $z=Re^{i\theta}$ 
and change variables in the above integral 
by substituting $R^2 \, d\theta = i z \, d\conj{z}$. Next we 
apply the Cauchy-Green theorem to the resulting integral.  
After rearrangement, 
we see that  
\begin{equation}
\begin{split}\label{eq:base}
\frac{R^2}{2} &\int_0^{2\pi} |f(Re^{i\theta})|^p \wea(R^2)\, d\theta -
\int_{\D_R} |z|^2 |f(z)|^p w'(|z|^2) \, dA  \\ 
=&\int_{\D_R} \left(\frac{p}{2}z f'(z) + f(z)\right) |f(z)|^{p-1} 
(\sgn \conj{f(z)})  \wea(|z|^2) \, dA
\end{split}
\end{equation}
Note that the left-hand side of equation \eqref{eq:base} is 
non-negative, since the first integral in the expression is non-negative 
and the second integral in the expression is non-positive.  This 
is due to the assumption that $\wea$ is decreasing.  

Now consider the right hand side of equation \eqref{eq:base}.
Let $k$ be a fixed function analytic in ${\D_R}$ and in 
$C^1(\overline{\D})$.  
Let $f_n$ be the solution to the extremal problem of maximizing 
the real part of 
$\int_{\D_R} g(z) \overline{k}(z) \, \wea(|z|^2) \, dA(z)$ over all
polynomials $g$ of degree at most $n$ such that 
$\|g\|_{A^p_R(\wea(|z|^2))} = 1$. 
Call the maximum $\|k\|^*_n$. 
By Theorem \ref{thm:exint} applied to the space of polynomials of 
degree $n$ considered as a subspace of $A^p_R(\wea(|z|^2))$, 
we have 
\[
\begin{split}
&\int_{\D_R} \left(\frac{p}{2} zf_n'(z)+f_n(z)\right) |f_n(z)|^{p-1} 
(\sgn \conj{f_n(z)})  \wea(|z|^2) \, dA \\
= \frac{1}{\|k\|^*_n} &\int_{\D_R} \left(\frac{p}{2} zf_n'(z)+f_n(z)\right) 
\overline{k}(z) \wea(|z|^2) \, dA,
\end{split}
\]
since $zf_n'(z)$ is also a polynomial of degree $n$.

If we take equation \eqref{eq:base} with $f_n$ in place of $f$, 
and use the above equation and also use the fact that
\[
  zf_n'(z)\wea(|z|^2) =
 \partial_z [zf_n(z) \wea(|z|^2)] - zf_n(z)\wea'(|z|^2) \conj{z} 
- f_n(z)\wea(|z|^2),
\]
we see that 
\[
\begin{split}
\frac{R^2}{2} &\int_0^{2\pi} |f_n(Re^{i\theta})|^p \wea(R^2)\, d\theta -
\int_{\D_R} |z|^2 |f_n(z)|^p \wea'(|z|^2) \, dA  \\ 
= 
\frac{p}{2\|k\|^*_n} &\int_{\D_R}  
\partial_z [zf_n(z) \wea(|z|^2) \conj{k(z)} ] \, dA -
\frac{p}{2\|k\|^*_n} \int_{\D_R}   |z|^2f_n(z)\conj{k(z)} \wea'(|z|^2) \, dA \\
+  \frac{1}{\|k\|^*_n} & \left(1-\frac{p}{2}\right) \int_{\D_R} f_n(z) 
\conj{k(z)} \wea(|z|^2) \, dA
\\
= \frac{p}{2\|k\|^*_n} &\int_{\D_R}  
\partial_z [zf_n(z) \wea(|z|^2) \conj{k(z)} ] \, dA -
\frac{p}{2\|k\|^*_n} \int_{\D_R}   |z|^2f_n(z)\conj{k(z)} \wea'(|z|^2) \, dA \\
+ \frac{1}{\|k\|^*_n} &\left(1-\frac{p}{2}\right) \left\{
\int_{\D_R} 
\partial_{\conj{z}} \left[f_n(z) \conj{z K(z)} \wea(|z|^2)\right] \, dA - 
\int_{\D_R} |z|^2 f_n \conj{K} \wea'(|z|^2)\, dA \right\}
\end{split}
\]
where $K(z)= (1/z) \int_0^z k(\zeta)\, d\zeta$. 
Applying the Cauchy-Green theorem again and changing the 
variable of integration to $\theta$ in the integrals over the boundary 
of the disc shows that 
\begin{equation}\label{eq:exdisc}
\begin{split}
\frac{R^2}{2} &\int_0^{2\pi} |f_n(Re^{i\theta})|^p \wea(R^2)\, d\theta -
\int_{\D_R} |z|^2 |f_n(z)|^p \wea'(|z|^2) \, dA 
= \\
\frac{p}{2\|k\|^*_n} \frac{R^2}{2} &\int_0^{2\pi}  f_n(Re^{i\theta})  
\conj{k(Re^{i\theta})} \wea(R^2)  \, d\theta -
\frac{p}{2\|k\|^*_n} \int_{\D_R}   |z|^2 f_n(z)\conj{k(z)} \wea'(|z|^2)   \, dA \\
+ & \frac{1}{\|k\|^*_n} \left(1-\frac{p}{2}\right) 
\left\{ \frac{R^2}{2}\int_{0}^{2\pi} f_n(z) \conj{K(z)}  \wea(R^2) \, d\theta
- \int_{\D_R} |z|^2 f_n \conj{K} \wea'(|z|^2) \, dA \right\}
\end{split}
\end{equation}

Now, define the $p^{\textrm{th}}$ integral mean of 
an analytic function $f$ at radius $r<R$ by  
\[
{M_p(r,f)} = 
\left\{\int_0^{2\pi} |f(re^{i\theta})|^p \, d\theta \right\}^{1/p}
\] 
and define $M_p(R,f) = \lim_{r \rightarrow R^{-}} M_p(r,f).$ 
Note that this differs by a factor of $(2\pi)^{-1/p}$ from the 
usual definition. 
For $0 < r \le R$, let  
\begin{equation}
{D_p(r,f;\wea)} = 
\left\{ -\int_{\D_r} |z|^2 |f(z)|^p \wea'(|z|^2) \, dA \right\}^{1/p} .
\end{equation}

We write $D_p(r,f)$ for $D_p(r,f;\wea)$ when it is clear what the 
function $\wea$ is.  
It is clear that ${D_p(r,f)}$ is non-decreasing with $r$, and it is well 
known that the same is true for $M_p(r,f)$ (see \cite{D_Hp}, p.~9). 
We note in passing that
in at least one case $D_p(r,f)$ can be given a physical interpretation.
A function $f$ in the Fock space for 
$p=2$ can represent the state of 
a quantum harmonic oscillator, in which case $D_2(\infty, f)$ represents 
a quantity related to the expected energy of the oscillator.

Let $q$ be the conjugate exponent to $p$, so that $1/p + 1/q = 1$. 
Also, note that 
\[
\begin{split} 
M_q(r,zK) &=  \left \{ \int_{0}^{2\pi} |re^{i\theta}K(re^{i\theta})|^q \, d\theta 
\right\}^{1/q} = 
 \left \{ \int_{0}^{2\pi} \left|\int_0^r k(\rho e^{i\theta}) e^{i\theta} 
\, d\rho \right|^q  \, d\theta \right\}^{1/q} \\ &\le 
 \int_{0}^{r} \left\{ \int_0^{2\pi} 
|k(\rho e^{i\theta})|^q \, d\theta  \right\}^{1/q} \,d\rho 
= \int_0^r M_q(\rho,k) \, d\rho \le r M_q(r,k).
\end{split}
\]
Thus $M_q(r, K) \le M_q(r, k)$, which also implies that 
$D_q(r,K) \le D_q(r,k)$ since the measure $|z|^2 \wea'(|z|^2)$ is a 
radial measure. 

Let $\widehat{p} = \max(p-1,1)$. 
Returning to equation \eqref{eq:exdisc} and using H\"{o}lder's inequality, 
we see that 
\[
\begin{split}
\frac{R^2}{2} \wea(R^2) & {M_p^p(R,f_n)} + {D_p^p(R,f_n)} 
\\
&\le \frac{1}{\|k\|^*_n}  
\left\{
\frac{p}{2} \frac{R^2}{2}\wea(R^2){M_p(R,f_n)}{M_q(R,k)} + 
\frac{p}{2}{D_p(R,f_n)}{D_q(R,k)} \right. \\
&\quad
\left. + \left|1-\frac{p}{2}\right| \left[
\frac{R^2}{2}\wea(R^2){M_p(R,f_n)}{M_q(R,K)} 
+ D_p(R,f_n) D_q(R, K) \right]
\right\}
\\
&\le
\frac{1}{\|k\|^*_n} \left\{
\widehat{p}
\frac{R^2}{2}\wea(R^2){M_p(R,f_n)}{M_q(R,k)} + 
\widehat{p} {D_p(R,f_n)}{D_q(R,k)} \right\}.
\end{split}
\]
For ease of notation, 
define $N_p(r, g) = (r^2/2)^{1/p} \wea(r^2)^{1/p} M_p(r,g)$ for 
any analytic function $g$.  
Then the right side of the last displayed inequality is 
at most 
\[
\begin{split}
&\frac{\widehat{p}}{\|k\|^*_n}
 \left[ \left(
 \frac{R^{2}}{2} \wea(R^2)\right)^{1/p}  {M_p(R,f_n)} + {D_p(R,f_n)}
 \right] \times 
\\
& \qquad \qquad \qquad
 \left[ \left(
 \frac{R^{2}}{2} \wea(R^2)\right)^{1/q}  {M_q(R,k)} + {D_q(R,k)}
 \right] \\
&=\frac{\widehat{p}}{\|k\|^*_n}
\Big[ N_p(R,f_n) + {D_p(R,f_n)}\Big]
\Big[ 
N_q(R,k_n)+ {D_q(R,k)}
\Big] \\
&\le \frac{2^{1/q} \widehat{p}}{\|k\|^*_n}
\Big[N_p^p(R, f_n) + {D_p^p(R,f_n)}
\Big]^{1/p}
\Big[  {N_q(R,k)} + {D_q(R,k)}
\Big].
\end{split}
\]
And thus we have
\begin{equation*}
\Big(
N_p^p(R,f_n)+ {D_p^p(R,f_n)} 
\Big)^{1/q}
\le 
\frac{2^{1/q} \widehat{p}}{\|k\|^*_n}
\Big[ 
N_q(R,k) + {D_q(R,k)}
\Big].
\end{equation*}
If $r < R$, this implies that
\begin{equation}\label{eq:ineqdisc1a}
{N_p^p(r,f_n)} + {D_p^p(r,f_n)} 
\le
\frac{2^{1/q} \widehat{p}}{\|k\|^*_n}
\Big[
{N_q(R,k)} + {D_q(R,k)}
\Big]^q .
\end{equation}

Now let $f$ denote the solution of our extremal problem over 
the full space.  
Observe that as $n \rightarrow \infty$, we have 
$f_n \rightarrow f$ in $A^p_R(\wea(|z|^2))$ and 
$\|k\|_n^* \rightarrow \|k\|^*$ by 
Theorem \ref{thm:extspacecont}. 
Thus $f_n \rightarrow f$ uniformly on compact subsets of $\D_R$
by Proposition \ref{prop:apbanach}.   
So $M_p(r,f_n) \rightarrow M_p(r,f)$ and 
$D_p(r, f_n) \rightarrow D_p(r, f)$ as $n \rightarrow \infty$.
Also recall that $M_p(r,f)$ and $D_p(r,f)$ are increasing with $r$. 
Thus, in inequality \eqref{eq:ineqdisc1a}, if we  
let first $n \rightarrow \infty$ and then 
$r \rightarrow \infty$, we have 
\begin{equation}\label{eq:ineqdisc1}
{N_p^p(R,f)} + {D_p^p(R,f)} 
\le
\frac{2^{1/q} \widehat{p}}{\|k\|^*}
\Big[
{N_q(R,k)} + {D_q(R,k)}
\Big]^q .
\end{equation}

Now suppose that $k$ is not in $C^1(\close{D})$  
but that $M_q(R,k) < \infty$.  It is well known that there is a sequence 
of polynomials $k_n$ such that $M_q(R, k-k_n) \rightarrow 0$ as 
$n \rightarrow \infty$ (this follows from Theorem 2.6 in 
\cite{D_Hp}). 
Now since $M_q(r,g)$ increases with $r$ for any analytic function 
$g$, and since 
\[
\begin{split}
D_q^q(R,k-k_n) &= 
- \int_0^R r^2 M_q^q(r,k-k_n) \wea'(r) r \, dr \\
 &\le 
\left(-\int_0^R \wea'(r^2) r\, dr \right) M_q^q(R, k-k_n) R^2 \\
&=
\frac{1}{2} ( \wea(0) - \wea(R^2) )  M_q^q(R, k-k_n) R^2 
\end{split}
\]
we have that $D_q(R,k-k_n) \rightarrow 0$ as $n \rightarrow \infty$.    
Thus, by Minkowski's inequality, we have that 
$M_q(R, k) - M_q(R,k_n)$ and $D_q(R,k) - D_q(R, k_n)$ both approach 
$0$ as $n \rightarrow \infty$. 
Now, for $r<R$ we have 
\[
{N_p^p(r,f_n)} + {D_p^p(r,f_n)} 
\le
\frac{2^{1/q} \widehat{p}}{\|k\|^*}
\Big[
{N_q(R,k_n)} + {D_q(R,k_n)}
\Big]^q.
\]
By Theorem \ref{thm:extkernelcont}, as $n\rightarrow \infty$ we have 
$f_n \rightarrow f$ in the $A^p_R(\wea(|z|^2))$ norm, and thus 
$f_n \rightarrow f$ uniformly in 
$\{ |z| \le r\}$ by Proposition \ref{prop:apbanach}.  
Therefore, $D_p(f_n,r) \rightarrow D_p(f,r)$ and 
$M_p(f_n,r) \rightarrow M_p(f,r)$.  Thus we have that 
\[
{N_p^p(r,f)} + {D_p^p(r,f)} 
\le
\frac{2^{1/q} \widehat{p}}{\|k\|^*}
\Big[ {N_q(R,k)} + {D_q(R,k)}
\Big]^q.
\]
Letting $r \rightarrow R$ shows that inequality 
\eqref{eq:ineqdisc1} still holds. 
We summarize our results in a theorem. 

\begin{theorem}\label{thm:extbounddisc}
Let $1 < p < \infty$ and let $0 < R < \infty$. 
Let the function $\wea$ be analytic in a neighborhood 
of $[0,R^2)$, and let $\wea$ be positive, non-increasing,  
and non-constant on $[0,R^2)$. 
Suppose that $f$ is the extremal function in 
$A^p_R(\wea(|z|^2))$ for the integral kernel $k$.  Then 
\[
\frac{R^2}{2} \wea(R^2) {M_p^p(R,f)} + {D_p^p(R,f)} 
\le
\frac{2^{1/q} \widehat{p}}{\|k\|^*}
\left[ \left(
\frac{R^{2}}{2} \wea(R^2)\right)^{1/q}  {M_q(R,k)} + {D_q(R,k)}
\right]^q.
\]
\end{theorem}

\section{Regularity of Extremal Functions in Fock type Spaces}\label{sec:regplane}

In this section we consider regularity of extremal functions in the 
Fock type spaces $A^p_\infty(\nu)$. 
The measures we consider are those which satisfy our previous
assumptions and for which
$\lim_{r \rightarrow \infty} r^n \wea(r^2) = 0$ and 
$\lim_{r \rightarrow \infty} r^n \wea'(r^2) = 0$
for all integers $n$, and for which polynomials are dense in 
$A^p_\infty(\wea(|z|^2))$ and 
$A^q_\infty(\wea(|z|^2) - |z|^2 \wea'(|z|^2))$, where $q$ is 
the conjugate exponent to $p$. 
In Section \ref{sec:polydense}, we give some sufficient 
conditions for polynomials to be dense in these spaces. 
For example, for $1 < p < \infty$ polynomials are dense in the Fock space 
(for which $\wea(z) = e^{-\alpha z}$).

Recall that equation \eqref{eq:base} says, for 
$f$ in $C^1(\close{\D_R})$ and analytic in $\D_R$, and for 
$0 < R < \infty$, that 
\[
\begin{split}
\frac{R^2}{2} &\int_0^{2\pi} |f(Re^{i\theta})|^p \wea(R^2)\, d\theta -
\int_{\D_R} |z|^2 |f(z)|^p w'(|z|^2) \, dA  \\ 
=&\int_{\D_R} (\frac{p}{2} z f'(z) + f(z)) |f(z)|^{p-1} 
(\sgn \conj{f(z)})  \wea(|z|^2) \, dA.
\end{split}
\]
Suppose that $f$ is a polynomial. 
Then letting $R \rightarrow \infty$ in equation \eqref{eq:base} 
gives 
\[
\begin{split}
-\int_{\mathbb{C}} |z|^2 |f(z)|^p w'(|z|^2) \, dA 
=\int_{\mathbb{C}} \left(\frac{p}{2} z f'(z) + f(z)\right) |f(z)|^{p-1} 
(\sgn \conj{f(z)})  \wea(|z|^2) \, dA.
\end{split}
\]

Consider the extremal problem for the space $A^p_\infty(\wea(|z|^2))$ 
with kernel $k$, where $k$ is a polynomial. 
Denote the solution to the extremal problem \eqref{eq:extprob} 
over polynomials of 
degree at most $n$ by $f_n$. 

As before, the right hand side of equation \eqref{eq:base} equals
\[
\frac{1}{\|k\|^*_n} \int_{\mathbb{C}} \left(\frac{p}{2} zf_n'(z)+f_n(z)
\right) 
\conj{k(z)} \wea(|z|^2)\, dA
\]
and the same computation as before gives
\[
\begin{split}
-&\int_{\mathbb{C}} |z|^2 |f_n(z)|^p \wea'(|z|^2) \, dA 
= \\
&\lim_{R \rightarrow \infty}
\frac{1}{\|k\|^*_n}
\left\{
\frac{p}{2} \int_{\D_R}  
\partial_z [zf_n(z) \wea(|z|^2) \conj{k(z)} ] \, dA -
\frac{p}{2} \int_{\D_R}   |z|^2f_n(z)\conj{k(z)} \wea'(|z|^2) 
\, dA \right. 
\\
+ & \left. 
\left(1-\frac{p}{2}\right) \int_{\D_R} 
\left[
\partial_{\conj{z}} \left[f_n(z) \conj{z K(z)} \wea(|z|^2)\right] \, dA
 - \int_{\D_R} |z|^2 f_n \conj{K} \wea'(|z|^2) \, dA
\right]
\right\} , 
\end{split}
\]
where $K(z) = (1/z) \int_0^z k(\zeta) \, d\zeta$ as before.  
We now use the Cauchy-Green theorem to see that 
\[
\begin{split}
-&\int_{\mathbb{C}} |z|^2 |f_n(z)|^p \wea'(|z|^2) \, dA 
= \\
\lim_{R \rightarrow \infty} & \frac{1}{\|k\|^*_n} \left\{
\frac{p}{2}  \frac{i}{2} \int_{\partial \D_R}  zf_n(z)  \conj{k(z)}  
\, \wea(|z|^2) d\conj{z} -
\frac{p}{2} \int_{\D_R}   |z|^2 f_n(z)\conj{k(z)} \wea'(|z|^2) \, dA   
\right. \\
+ & \left. \left(1-\frac{p}{2}\right) 
\left[
\frac{1}{2i} \int_{\D_R} 
f_n(z) \conj{K(z)} \wea(|z|^2) \conj{z} \, dz
 - \int_{\D_R} |z|^2 f_n \conj{K} \wea'(|z|^2) \, dA(z)
\right]
\right\}.
\end{split}
\]
Since $f_n$ and $k$ are polynomials, our assumptions on $\wea$ imply that
\[
\begin{split}
-\int_{\mathbb{C}} |z|^2 |f_n(z)|^p \wea'(|z|^2) \, dA &=
-\frac{p}{2\|k\|^*_n} \int_{\mathbb{C}}   |z|^2 f_n(z)\conj{k(z)} 
\wea'(|z|^2) \, dA  \\ 
& \qquad - \left(1 - \frac{p}{2}\right) \int_{\D_R} |z|^2 f_n \conj{K} 
\wea'(|z|^2) \, dA(z).  
\end{split}
\]
Applying H\"{o}lder's inequality, we see that 
\[
\begin{split}
D_p^p(\infty,f_n) &\le \frac{p}{2\|k\|^*_n}  D_p(\infty, f_n) D_q(\infty, k) + 
\frac{1}{\|k\|^*_n} 
\left| 1 - \frac{p}{2} \right| D_p(\infty, f_n) D_q(\infty, K)
\\ &\le 
\frac{\widehat{p}}{\|k\|^*_n}  D_p(\infty, f_n) D_q(\infty, k),
\end{split}
\]
so that
\[
D_p(\infty,f_n) \le \left[ \frac{\widehat{p}}{\|k\|^*_n} 
D_q(\infty, k) \right]^{1/(p-1)}.
\]
Therefore,
\[
D_p(R,f_n) \le \left[ \frac{\widehat{p}}{\|k\|^*_n} D_q(\infty, k) \right]^{1/(p-1)},
\]
where $R > 0$ is arbitrary.

Let $f$ denote the solution to the extremal problem 
over the full space. 
Then $f_n \rightarrow f$  in $A^p_\infty(\wea(|z|^2))$ 
and $\|k\|^*_n \rightarrow \|k\|^*$
by Theorem \ref{thm:extspacecont}. Also, by 
Proposition \ref{prop:apbanach}, 
$f_n \rightarrow f$ uniformly on $|z| \le R$. 
Letting $n \rightarrow \infty$ and then letting $R \rightarrow \infty$ 
in the above displayed inequality gives 
\[
D_p(\infty,f) \le \left[ \frac{\widehat{p}}{\|k\|^*} D_q(\infty, k) \right]^{1/(p-1)}.
\]

Lastly, if $k$ is not a polynomial, we let $k_n$ 
be as sequence of polynomials approaching $k$ in 
$A^q_\infty(\wea(|z|^2))$, such that 
$D_q(\infty, k_n) \rightarrow D_q(\infty, k)$ as 
$n \rightarrow \infty$. This can be done since polynomials 
are dense in the space $A^q_\infty(\wea(|z|^2) - |z|^2 \wea'(|z|^2))$. 
Let $f_n$ be the solution to the extremal problem with kernel $k_n$. 
Then we have, for fixed $R > 0$, that 
\[
D_p(R,f_n) \le \left[ \frac{\widehat{p}}{\|k_n\|^*} D_q(\infty, k_n) \right]^{1/(p-1)}.
\]
Now $f_n \rightarrow f$ 
in $A^p_\infty(\omega(|z|^2))$ by Theorem \ref{thm:extkernelcont},
and thus uniformly for $|z| \le R$ by Proposition \ref{prop:apbanach}. Also,  
$\|k_n\|^* \rightarrow \|k\|^*$ as $n \rightarrow \infty$, 
since $\|k - k_n\|^*$ is bounded above by  
$\|k-k_n\|_{A^q_\infty(\wea(|z|^2))}$, which approaches $0$.  Therefore, we have 
that 
\[
D_p(R,f) \le \left[ \frac{\widehat{p}}{\|k\|^*} D_q(\infty, k) \right]^{1/(p-1)}.
\]
Letting $R \rightarrow \infty$ gives 
\[
D_p(\infty,f) \le \left[ \frac{\widehat{p}}{\|k\|^*} D_q(\infty, k) \right]^{1/(p-1)}.
\]

Again, we state our results in a Theorem.
\begin{theorem}\label{thm:extboundplane} 
Let $1 < p < \infty$. 
Let the function $\wea$ be analytic in a neighborhood 
of $[0,\infty)$, and let $\wea$ be, positive, non-increasing,  
and non-constant on $[0,\infty)$.
Also suppose that 
$\lim_{r \rightarrow \infty} r^n \wea(r^2) = 
\lim_{r \rightarrow \infty} r^n \wea'(r^2) = 0$
for all integers $n$, and that the polynomials are dense in 
$A^p_\infty(\wea(|z|^2))$ and in 
$A^q_\infty(\wea(|z|^2) - |z|^2 \wea'(|z|^2))$.
 Suppose that $f$ is the extremal function in 
$A^p_R(\wea(|z|^2))$ for the integral kernel $k$.  Then 
\[
D_p(\infty,f) \le \left[ \frac{\widehat{p}}{\|k\|^*} D_q(\infty, k) \right]^{1/(p-1)}.
\]
\end{theorem}

One could question whether the condition 
$D_p(\infty,f) < \infty$ is implied by the condition 
$f \in A^p_\infty(\wea(|z|)^2)$, in which case the above theorem 
would be less interesting.  
However, in general $f \in A^p_\infty(\wea(|z|)^2)$ does not 
imply $D_p(\infty,f) < \infty$.
For example, consider the Fock space with measure $e^{-\alpha |z|^2}$.  
In this case, the statement that $D_p(\infty,f) < \infty$ is equivalent 
to $f$ being in the space $A^p_\infty(|z|^2 e^{-\alpha|z|^2})$. 
Now, the norm of 
$z^n$ in the original Fock space is 
\[
\left[ 
\frac{\pi}{\alpha^{np/2}} \Gamma \left( \frac{np}{2}+1 \right)
\right]^{1/p}, 
\]
while its norm in the second space is 
\[
\left[
\frac{\pi}{\alpha^{np/2+1}} \Gamma \left( \frac{np}{2}+2 \right)
\right]^{1/p} .
\]
The ratio of the second norm to the first is $((np + 2)/(2\alpha))^{1/p}$, 
which is unbounded in $n$.  If every element in the Fock space were in 
$A^p_\infty(|z|^2 e^{-\alpha|z|^2})$, then 
by the closed graph theorem the identity map
from the Fock space into $A^p_\infty(|z|^2 e^{-\alpha|z|^2})$ 
would be bounded, which contradicts the above analysis of the 
norms of the monomials. 
Thus, $f \in A^p_\infty(e^{-\alpha |z|^2})$ does not imply that 
$f \in A^p_\infty(|z|^2 e^{-\alpha |z|^2})$. 

\section{Growth of Integral Means and Log-Convex Functions}\label{sec:logconvex}

Theorem \ref{thm:extboundplane} does not 
bound the quantity 
$\lim_{r \rightarrow \infty} r^2 \wea(r^2) M_p^p(r,f)$, even though a 
similar term was bounded in Theorem \ref{thm:extbounddisc}.  
Thus, by analogy with Theorem \ref{thm:extbounddisc}, 
it is natural to ask whether 
$\wea(r^2) M_p^p(r,f)$ is bounded as $r \rightarrow \infty$
if $r^2 \wea(r^2) M_q^q(r,k)$ is bounded as $r \rightarrow \infty$ 
and if $D_q(\infty,k)$ is bounded.  For certain measures, we show 
that this is the case.  In fact, if $D_q(\infty,k)$ is bounded, so 
is $D_p(\infty,f)$ by Theorem \ref{thm:extboundplane}, and
in this section we show that for certain measures 
$\wea$, the condition $D_p(f,\infty) < \infty$ implies that 
$r^3 \wea(r^2) M_p^p(r,f) \rightarrow 0$ as $r \rightarrow \infty$ for 
any entire function $f$. 

It will simplify matters if we introduce some notation. 
Let $\lambda(x)$ be a positive, increasing, smooth function 
defined for $x \ge R$, where $R \ge 0$.  
Now, let $X = \log x$ and $Y = \log y$.  Define 
$\nu(X) = \log(\lambda(x)) = \log(\lambda(e^X))$. 
Let $g(x)$ be differentiable and a log-convex function of $\log x$, 
i.e~let $\log g(x) = \log g(e^X)$ be a convex function of $X$. 
Let $\widetilde{g}(X) = g(e^X)$ and $\widetilde{\nu}(X) = \nu(e^X)$.

Now suppose that for some $x_1 > 0$ we have 
$g(x_1) = \lambda(x_1)$ but that 
$g(x) \le \lambda(x)$ for some $x < x_1$. 
Then for some $x_0$ such that $x < x_0 \le x_1$ we must have 
$
g(x_0) = \lambda(x_0)
$ 
and $g'(x_0) \ge \lambda'(x_0)$, which implies that 
\[
\frac{d \log \widetilde{g}(X)}{dX} (X_0) \ge 
\frac{d \widetilde{\nu}}{dX}(X_0).
\]
Now let 
\[
Y = \widetilde{\nu}(X_0) + \left(\frac{d \widetilde{\nu}}{dX}(X_0) \right) 
\cdot (X-X_0).
\]  
For $X \ge X_0$, the function $Y$ lies below the line tangent to the 
function $\log \widetilde{g}(X)$ at $X_0$.  
Since $\log \widetilde{g}(X)$ is a convex function of $X$, this means that 
$Y \le \log \widetilde{g}(X)$ for all $X \ge X_0$. 

We compute that 
\[
\frac{d\widetilde{\nu}}{dX} = \frac{d \log \lambda(e^X)}{dX} = 
\frac{e^X \lambda'(e^X)}{\lambda(e^X)} = \frac{x \lambda'(x)}{\lambda(x)}.
\]
Let $e^Y = y$.  Then we have 
\[
y = \lambda(x_0) \left( \frac{x}{x_0} \right)^{x_0 \lambda'(x_0)
  / \lambda(x_0)}.
\]
Now let 
\[
S(x_0, \lambda) = \int_{x_0}^\infty y \lambda(x)^{-1}\, dx = 
\int_{x_0}^\infty  
    \frac{\lambda(x_0)}{\lambda(x)} \left( \frac{x}{x_0} \right)^{x_0 \lambda'(x_0)
  / \lambda(x_0)}
\, dx.
\]
Note that $S(x_0,\alpha \lambda) = S(x_0, \lambda)$ for 
$\alpha \ne 0$. 
From the discussion above, we have the following proposition.
\begin{lemma}
Let $g$ and $\lambda$ be as above. 
Suppose that there is some $x_1$ such that $g(x_1)=\lambda(x_1)$ and that 
there is some $x_2$ such that $R \le x_2 < x_1$ and such that 
$g(x_2) < \lambda(x_2)$.  Then there is some $x_0$ such that 
$x_2 < x_0 \le x_1$ and $g(x_0) = \lambda(x_0)$, and such that 
\[
\int_{x_0}^\infty  g(x) \lambda(x)^{-1} \, dx \ge S(x_0, \lambda).
\]
\end{lemma}

We now show that if 
\[
\int_0^\infty  g(x) \lambda(x)^{-1}\, dx < \infty
\]
and if $\liminf_{x\rightarrow \infty} S(x, \lambda) > 0$, we have
$\lim_{x \rightarrow \infty} g(x) \lambda(x)^{-1} = 0$. 
Suppose, for the sake of contradiction, that 
$\limsup_{x \rightarrow \infty} g(x) \lambda(x)^{-1} = 2k >0$.  
From the fact that the 
integral above is finite we must have 
$\liminf_{x \rightarrow \infty} g(x) \lambda(x)^{-1} = 0$. 
Thus, there must be some sequence of points 
$a_1 < b_1 < a_2 < b_2 \cdots$ such that 
$g(a_j) \lambda(a_j)^{-1} < k$ and 
$g(b_j) \lambda(b_j)^{-1} = k$, and such that $b_n \rightarrow \infty$ 
and $a_n \rightarrow \infty$ as $n \rightarrow \infty$.
Thus if we define $\lambda_k = k \lambda$, the previous proposition  
shows that there is some sequence of points $x_n \rightarrow \infty$ 
such that 
\[
\int_{x_n}^\infty  g(x) \lambda_k(x)^{-1} \, dx \ge 
S(x_n, \lambda_k)
\]
for all $n$. As $n \rightarrow \infty$, the left hand side of the previous
inequality must approach $0$.  So if we can show for all $k>0$ 
that $\liminf_{x \rightarrow \infty} S(x, \lambda_k) > 0$, we will
obtain a contradiction, showing in fact that $k = 0$. But
$S(x, \lambda_k) = S(x, \lambda)$ for $k \ne 0$.
Thus, we have obtained a contradiction between our assumptions 
and the supposition that $k \ne 0$.  In summary, we have the 
following theorem. 

\begin{theorem}\label{thm:logconvexlimit}
Suppose that $\liminf_{x \rightarrow \infty} S(x, \lambda) > 0$.  Then 
if $g(x)$ is a log-convex function of 
$\log x$ such that 
\[
\int_0^\infty  g(x) \lambda(x)^{-1} \, dx < \infty,
\]
we have that $\lim_{x \rightarrow \infty} g(x) \lambda^{-1}(x) = 0$. 
\end{theorem}

To apply this theorem in our situation, we choose
$\lambda(x) = -1/\wea'(x^2)$. 
Recall that if $f$ is an analytic function then
 $M_p(r,f)$ is a log convex function of $\log r$ 
(see \cite{D_Hp}, p.~9).  Now, 
let $g(x) = x^3 M_p^p(x,f)$.  Letting $X = \log x$, we have 
\[
\log g(e^X) = 3X + p \log M_p(e^X, f).
\]
Since both $X$ and $\log M_p(e^X,f)$  are convex functions of $X$, so is 
$g(x)$.  
Suppose that 
\[
D_p^p(\infty,f;\omega) = 
\int_{\mathbb{C}} |z|^2 |f(z)|^p \lambda(|z|)^{-1} < \infty,
\]
which means that 
\[
\int_0^\infty x^3 M_p^p(x,f) \lambda(x)^{-1} = 
\int_0^\infty  g(x) \lambda(x)^{-1}  
< \infty.
\]
If $\liminf_{x \rightarrow \infty} S(x, \lambda) > 0$, then 
by Theorem \ref{thm:logconvexlimit} we have 
$\lim_{r \rightarrow \infty} r^3 M_p^p(r,f) \wea'(r^2) = 0$.  
Thus we have the following theorem.
\begin{theorem}
Suppose that 
$\liminf_{x \rightarrow \infty} S(x, -1/\wea'(r^2)) > 0$
and that there is some positive constant $C$ such that 
$-\wea'(r) \ge C \wea(r)$ for all sufficiently large $r$.
If $D_p(\infty, f; \omega) < \infty$, then
$\lim_{r \rightarrow \infty} r^3 M_p^p(r,f) \wea(r^2) = 0$.  
\end{theorem}

{\bf Example}.  Suppose that $\wea(|z|^2) = (1/\alpha) e^{-\alpha |z|^2}$, 
so that we are dealing with a Fock space.  (The 
$1/\alpha$ factor is there as a convenience 
so an extra factor of $\alpha$ does not 
appear in the definition of $\lambda$.)
Note that 
\[
\begin{split}
S(x_0, \lambda) &= 
\int_{x_0}^\infty e^{\alpha x_0^2} e^{-\alpha x^2} 
    \left(\frac{x}{x_0}\right)^{2\alpha x_0^2} \, dx \\ &= 
e^{\alpha x_0^2}(\alpha x_0^2)^{-\alpha x_0^2} 
       \int_{x_0}^\infty e^{-\alpha x^2} (\alpha x^2)^{\alpha x_0^2}\, dx \\
&= 
(1/2\alpha^{1/2}) e^{\alpha x_0^2}(\alpha x_0^2)^{-\alpha x_0^2} 
\int_{\alpha x_0^2}^\infty e^{-u} u^{\alpha x_0^2-(1/2)}\, du \\
&= 
(1/2\alpha^{1/2}) e^{\alpha x_0^2}(\alpha x_0^2)^{-\alpha x_0^2} 
\Gamma(\alpha x_0^2+(1/2), \alpha x_0^2) \\ 
& \ge 
(1/2\alpha^{1/2}) e^{\alpha x_0^2}(\alpha x_0^2)^{-\alpha x_0^2} 
\Gamma(\alpha x_0^2+(1/2), \alpha x_0^2+(1/2)),
\end{split}
\]
where $\Gamma$ is the incomplete Gamma function, as defined in 
\cite{NIST:DLMF}, formula 8.2.2. 
Now for large $x$, we have (\cite{NIST:DLMF}, formula 8.11.12)
\[
\Gamma(x,x) \sim x^{x}e^{-x} x^{-1/2} \sqrt{\pi/2}.
\]
which implies, for large enough $x$, that
\[
\Gamma(x+(1/2),x+(1/2)) \ge 
C \left(x+\frac{1}{2}\right)^{x+(1/2)} e^{-x-(1/2)} 
\left(x+\frac{1}{2}\right)^{-1/2}
\ge \frac{C x^{x}e^{-x}}{e^{1/2}}
\]
for some constant $C > 0$. 
So we have that 
\[
S(x_0, \lambda) \ge 
\frac{C}{2\alpha^{1/2}e^{1/2}} 
e^{\alpha x_0^2}(\alpha x_0^2)^{-\alpha x_0^2} 
(\alpha x_0^2)^{\alpha x_0^2} e^{-\alpha x_0^2} = 
\frac{C}{2\alpha^{1/2}e^{1/2}} 
\]
for large enough $x_0$.  
Thus, we have the following result, which we state as a theorem. 

\begin{theorem}
Let $f$ be an entire function and let $0<p<\infty$.  If
\[
\int_{\mathbb{C}} |f(z)|^p |z|^2 e^{-\alpha |z|^2} \, dA(z) < \infty,
\]
then 
\[
\lim_{r \rightarrow \infty} r^3 M_p^p(r,f) e^{-\alpha r^2} = 0.
\]
\end{theorem}

\section{Density of Polynomials in Various Spaces}\label{sec:polydense}

In this section, we discuss the density of polynomials in various 
weighted Bergman spaces.  
Propositions \ref{prop:apbanach}, \ref{prop:apdilationdense}, 
\ref{prop:appolydensedisc} and 
\ref{prop:a2polydense} are well known, at least in certain cases, and 
we follow the standard proofs.  
The proof of Proposition \ref{prop:appolydenseplane} is similar to the proof 
that the polynomials are dense in the Fock space (see e.g.~\cite{Zhu_Fock}). 

We say a nonnegative function $\nu$ defined on $[0,R)$ 
is a radial weight function 
if the measure $\nu(|z|)\, dA$ is in $L^p(\mathbb{D}_R)$, 
and it is not the case that $\nu = 0$ a.e.

\begin{prop}\label{prop:apbanach}
Let $0 < R \le \infty$. 
Suppose that $\nu$ is a radial weight function, 
and that there is some $R'$ such that $0 \le R' < R$ and  
for each $y$ such that $R' <y <R$, the quantity 
$\inf\{\web(x): R' \le x < y\}$ is nonzero.  Then 
$A^p_R(\web(|z|))$ is a Banach space, and convergence in the norm 
of $A^p_R(\nu(|z|))$ implies uniform convergence on compact 
subsets of $\D_R$. Also, the point evaluations of $A^p_R(\nu(|z|))$
are bounded uniformly on compact subsets of $\mathbb{D}_R$. 
\end{prop}
\begin{proof}
We first consider the case $R < \infty$. 
Since the space in question is a subspace of $L^p(\web(|z|))$, 
to show that it is a Banach space we 
need only show that it is closed.
To show this, we will show that convergence in the norm implies 
uniform convergence on compact subsets.  
Let $f \in A^p_R(\web(|z|))$.  Since $f$ is analytic, 
the function $|f(z)|^p$ is subharmonic, and so 
\[
|f(z)|^p \le \frac{1}{2\pi} \int_0^{2\pi} |f(z + re^{i\theta})|^p \, 
d \theta
\]
for any $z \in \D_R$ with $|z|+r < R$.   
Suppose first that $|z| > (R'+R)/2.$ 
Now let $r' = (R-|z|)/2$
and define
$m_z = \inf \{\web(|w|): |z|-r' < |w| < |z|+r'\}$. 
By assumption, $m_z > 0$. 
We multiply the last 
displayed inequality by $2\pi m_zr$ and integrate $r$ from $0$ to $r'$ to 
conclude that 
\[
\begin{split}
m_z |f(z)|^p \pi r'^2 &\le \int_{|z-w|<r'} |f(w)|^p m_z \, dA(w) \\ &\le 
\int_{|z-w|<r'} |f(w)|^p \web(|w|) \, dA(w) \\ &\le 
\int_{\D_R} |f(w)|^p \web(|w|) \, dA(w).
\end{split}
\]
This shows that point evaluation is a bounded linear functional for 
$ R > |z| > (R' + R)/2$, and the bound depends only on $|z|$. 
By the maximum principle, point evaluation is also bounded for 
$|z| \le (R + R')/2$.  
Since $m_z r'^2$ is a decreasing positive function of $|z|$, 
the bound is uniform on closed discs of radius less than $R$ centered at 
the origin. 
Thus, convergence in the norm implies convergence on 
compact subsets.  This implies that a convergent sequence of 
analytic functions converges to an analytic function, 
which shows the space is complete. 

For the case $R = \infty$, we repeat the above proof, except that 
we first consider $z$ such that $|z| > R' + 1$, and we let $r' = 1$. 
\end{proof}

We define the dilation of the analytic function $f$ by 
$f_\rho(z) = f(\rho z)$ for $0<\rho<\infty$. 

\begin{prop}\label{prop:apdilationdense}
Let $0 < R \le \infty$ and 
let $\nu$ be a radial weight function.
If $f$ is an analytic function in $A_R^p(\nu(|z|))$, then 
$f_\rho \rightarrow f$ in $A_R^p(\nu(|z|))$ as 
$\rho \rightarrow 1^{-}$. 
\end{prop}
\begin{proof}
Since the integral means of an analytic function are increasing, 
we have for $\rho < 1$ that $f_\rho \in A^p_R(\nu(|z|))$ and that
\[
M_p(r,f-f_\rho) \le M_p(r,f) + M_p(r,f_\rho) \le 2 M_p(r, f).
\]
The hypothesis that $f \in A_R^p(\nu(|z|))$ 
implies that $r M_p(r,f) \in L^p(\nu\, dr)$.  Thus 
\[
\int_0^R M_p^p(r,f-f_\rho) r \, \nu(r) \, dr \rightarrow 0
\]
as $\rho \rightarrow 1^-$, by the Lebesgue dominated convergence 
theorem. 
\end{proof}

\begin{prop}\label{prop:appolydensedisc}
Let $R < \infty$ and let $0 < p < \infty$. 
Assume that $\nu$ is a radial weight function.
Then the polynomials are dense in $A^p_R(\nu(|z|))$. 
\end{prop}
\begin{proof}
Let $f \in A_R^p(\nu(|z|))$ and let $\rho < 1$.
Since $f_\rho$ is analytic on $\D_{R/\rho}$, the Taylor series of 
$f_\rho$ converges to $f_\rho$ uniformly in $\D_{R}$, and thus 
each dilation is in the closure of the polynomials, since 
the integrability of $\nu(|z|)$ guarantees that uniform convergence on 
$\D_R$ implies convergence in $A_R^p(\nu(|z|))$.   
By Proposition \ref{prop:apdilationdense}, $f$ is in the closure of the 
set of its dilations.  Thus, $f$ is in the closure of the polynomials in 
$A_R^p(\nu(|z|))$. 
\end{proof}

The situation is more difficult for the case $R=\infty$.  We 
first prove the following proposition.

\begin{prop}\label{prop:a2polydense}
Let $\nu$ be a radial weight function on $[0,\infty)$
such that every polynomial is in $A_\infty^2(\nu(|z|))$.  
Then the polynomials are dense in $A_\infty^2(\nu(|z|))$. 
\end{prop}
\begin{proof}
Suppose $f$ is a function in 
$A_\infty^2(\nu(|z|))$ such that $\langle f, z^n \rangle = 0$ for every 
$n \ge 0$, and let $\sum_{n=0}^\infty a_n z^n$ be the Taylor 
series of $f$. We 
have by the dominated convergence theorem and 
H\"{o}lder's inequality that 
\[
\begin{split}
0 = \langle f, z^m \rangle &= 
\int_{\C} \left(\sum_{n=0}^\infty a_n z^n\right) \conj{z^m} \nu(|z|) \, dA(z) 
\\ &= 
\lim_{r \rightarrow \infty}
\int_{\D_r} \left(\sum_{n=0}^\infty a_n z^n\right) 
        \conj{z^m} \nu(|z|) \, dA(z), 
\end{split}
\]
for $m \ge 0$. 
By the uniform convergence of the Taylor series on $\mathbb{D}_r$ 
and the integrability of $\nu(|z|)$, the 
above expressions equal
\[ 
\begin{split}
&\quad
\lim_{r \rightarrow \infty}
\sum_{n=0}^\infty
\int_{\D_r}  a_n z^n
        \conj{z^m} \nu(|z|) \, dA(z) 
\\ &=
\lim_{r \rightarrow \infty} 
\int_{\D_r} a_m |z|^{2m} \nu(|z|) \, dA(z)
\\ &=
a_m \int_{\C} |z|^{2m} \nu(|z|) \, dA(z).
\end{split}
\]
But the above integral is positive, so $a_m = 0$ 
for each $m \ge 0$, and thus $f$ is identically $0$.  This shows that the 
polynomials are dense in $A_\infty^2(\nu(|z|))$. 
\end{proof}

For a radial weight function $\nu$ with the property that 
$A_\infty^p$ has bounded point evaluations (and thus is 
a Banach space), 
define $m_p(z;\nu) = \sup_{\|f\|_{A^p_\infty(\nu)} = 1} |f(z)|$. 
\begin{prop}\label{prop:appolydenseplane}
Let $\nu$ be a function on $[0,\infty)$
such that 
$A^p_\infty(\nu(|z|))$ contains every polynomial 
and has point evaluations uniformly bounded on compact subsets of 
$\mathbb{C}$. 
Suppose that for any $\rho$ such that  
$0<\rho<1$, 
there is some function $\mu$ 
such that $A^2_\infty (\mu(|z|))$ has point evaluations uniformly 
bounded on compact subsets of $\mathbb{C}$
and such that
\begin{equation}\label{eq:appolydense_frina2}
\int_{\C} m_p(\rho z;\nu)^{2} \mu(|z|) \, dA(z) =C_1^2 < \infty
\end{equation}
and
\begin{equation}\label{eq:appolydense_a2inap}
\int_{\C} m_2(z;\mu)^{p} \nu(|z|) \, dA(z) =C_2^p < \infty.
\end{equation}
Then the polynomials are dense in $A_\infty^p(\nu(|z|))$.
\end{prop}
\begin{proof}
Let $0 < \rho < 1$.
Note that
$M_p(r,f_\rho) = M_p(r\rho, f) \le M_p(r, f)$, so that 
$\|f_\rho\|_{A^p_\infty(\nu(|z|))} \le \|f\|_{A^p_\infty(\nu(|z|))}$.
Thus $|f(\rho z)| \le m_p(\rho z, \nu) \|f\|_{A_\infty^p(\nu(|z|))}$. 
Therefore $\|f_\rho\|_{A^2_\infty(\mu(|z|))}$ is at most 
$ C_1 \|f\|_{A^p_\infty(\nu(|z|))}$,  
and $f_\rho \in A^2_\infty(\mu(|z|))$.
Note that this implies that every polynomial is in 
$A^2_\infty(\mu(|z|))$, since if $p$ is a polynomial, then 
$p_{1/\rho}$ is also a polynomial and is in $A_\infty^p(\nu(|z|))$, which 
implies that $p$ is in $A^2_\infty(\mu(|z|))$. 

Now, by Proposition \ref{prop:a2polydense}, there is a sequence of polynomials 
in $A^2_\infty(\mu(|z|))$ that approach $f_\rho$.  But for any 
function $g \in A^2_\infty(\mu(|z|))$, we have 
$ \|g\|_{A^p_\infty(\nu(|z|))} \le C_2 \|g\|_{A^2_\infty(\mu(|z|))}$.   
Therefore, there is a sequence of polynomials 
approaching $f_\rho$ in $A^p_\infty(\nu(|z|))$. Since the functions 
$f_\rho$ approach $f$ in $A^p_\infty(\nu(|z|))$, there is a sequence of 
polynomials approaching $f$ in $A^p_\infty(\nu(|z|))$. 
\end{proof}
Note that the quantities $m_p(z, \nu)$ and $m_2(z, \mu)$ can often be estimated 
by the method used in the proof of Proposition \ref{prop:apbanach}.

The following corollary follows from Proposition \ref{prop:appolydenseplane}.
\begin{cor}
Suppose that $\nu$ is a nonzero decreasing function on 
$[0,\infty)$ such that 
$\nu(|z|) \in L^1(\mathbb{C})$ and 
such that every polynomial is in $A_\infty^p(\nu(|z|))$.
Also assume that 
for each $\rho$ such that $0<\rho<1$, there is a $\beta$ such 
that $0<\beta<1$ and such that 
\[
\int_{\mathbb{C}} \nu(\rho|z|+1)^{-2/p} \nu(|z|)^{2\beta/p} \, dA(z) < \infty
\]
and
\[
\int_{\mathbb{C}} \nu(|z|+1)^{-\beta} \nu(|z|) \, dA(z) < \infty.
\]
Then the polynomials are dense in $A_\infty^p(\nu(|z|))$. 
\end{cor}
\begin{proof}
Let $f$ be an entire function.  By subharmonicity, 
\[
|f(z)|^p \le \frac{1}{2\pi} \int_0^{2\pi} |f(z+r e^{i\theta})|^p \, d\theta
\]
for any $r > 0$. 
If we 
multiply the previous displayed inequality by $2 \pi r \nu(|z|+1)$ 
and integrate from $r=0$ to $r=1$, we find that 
\[
\begin{split}
\pi \nu(|z|+1) |f(z)|^p 
&\le \int_{|z-w|<1} |f(w)|^p \nu(|z|+1) \, dA(w) \\ 
&\le
\int_{|z-w|<1} |f(w)|^p \nu(|w|) \, dA(w) \\ &
\le 
\int_{\C} |f(w)|^p \nu(|w|) \, dA(w).
\end{split}
\]
Thus, we have that 
\[
|f(z)| \le \pi^{1/p} \nu(|z|+1)^{-1/p} \|f\|_{A^p_\infty(\nu(|z|))}.
\]
And similarly, 
\[
|f(z)| \le \pi^{1/2} \nu(|z|+1)^{-\beta/p} \|f\|_{A^2_\infty(\nu(|z|)^{2\beta/p})}.
\]
for any function $f \in A^2_\infty(\nu(|z|)^{2\beta/p})$.
So if there is some $\beta$ such that $0<\beta<1$ 
and such that
\[
\int_{\C} \nu(\rho|z|+1)^{-2/p} \nu(|z|)^{2\beta/p} \, dA(z) < \infty
\]
and
\[
\int_{\C} \nu(|z|+1)^{-\beta} \nu(|z|) \, dA(z) < \infty,
\]
then the result will hold by Proposition \ref{prop:appolydenseplane}.
\end{proof}

Note that if $\nu$ is a bounded function that is eventually decreasing, 
then $A^p_\infty(\nu(|z|))$ is equivalent in norm to 
$A^p_\infty(\widetilde{\nu}(|z|))$, where $\widetilde{\nu}$ is 
decreasing and $\widetilde{\nu}(x) = \nu(x)$ for $x$ sufficiently 
large.  Thus, the previous corollary can be applied in modified 
form to such functions $\nu$.

The next corollary is needed to apply the results of 
Section \ref{sec:regplane} to the Fock space. It follows from the 
above Corollary by choosing $\beta$ such that $\rho < \beta < 1$. 

\begin{cor}
Let $\alpha > 0$ and $0<p<\infty$. 
The space $A^p_\infty(|z|^2 e^{-\alpha |z|^2} + e^{-\alpha |z|^2})$ is a Banach 
space in which the polynomials are dense.
\end{cor}

\providecommand{\bysame}{\leavevmode\hbox to3em{\hrulefill}\thinspace}
\providecommand{\MR}{\relax\ifhmode\unskip\space\fi MR }
\providecommand{\MRhref}[2]{%
  \href{http://www.ams.org/mathscinet-getitem?mr=#1}{#2}
}
\providecommand{\href}[2]{#2}


\begin{thebibliography}{10}

\bibitem{Khavinson_nonvanishing}
Dov Aharonov, Catherine B{\'e}n{\'e}teau, Dmitry Khavinson, and Harold Shapiro,
  \emph{Extremal problems for nonvanishing functions in {B}ergman spaces},
  Selected topics in complex analysis, Oper. Theory Adv. Appl., vol. 158,
  Birkh\"auser, Basel, 2005, pp.~59--86. \MR{MR2147588 (2006i:30047)}

\bibitem{Benetau_Fock_Extremal}
Catherine B{\'e}n{\'e}teau, Brent~J. Carswell, and Sherwin Kouchekian,
  \emph{Extremal problems in the {F}ock space}, Comput. Methods Funct. Theory
  \textbf{10} (2010), no.~1, 189--206. \MR{2676450 (2011e:30124)}

\bibitem{Clarkson}
James~A. Clarkson, \emph{Uniformly convex spaces}, Trans. Amer. Math. Soc.
  \textbf{40} (1936), no.~3, 396--414. \MR{MR1501880}

\bibitem{NIST:DLMF}
\emph{{NIST Digital Library of Mathematical Functions}}, http://dlmf.nist.gov/,
  Release 1.0.6 of 2013-05-06, Online companion to \cite{Olver:2010:NHMF}.

\bibitem{DKSS_Mich}
P.~Duren, D.~Khavinson, H.~S. Shapiro, and C.~Sundberg, \emph{Invariant
  subspaces in {B}ergman spaces and the biharmonic equation}, Michigan Math. J.
  \textbf{41} (1994), no.~2, 247--259. \MR{MR1278431 (95e:46030)}

\bibitem{D_Hp}
Peter Duren, \emph{Theory of {$H\sp{p}$} spaces}, Pure and Applied Mathematics,
  Vol. 38, Academic Press, New York, 1970. \MR{MR0268655 (42 \#3552)}

\bibitem{DKS}
Peter Duren, Dmitry Khavinson, and Harold~S. Shapiro, \emph{Extremal functions
  in invariant subspaces of {B}ergman spaces}, Illinois J. Math. \textbf{40}
  (1996), no.~2, 202--210. \MR{1398090 (97h:30069)}

\bibitem{DKSS_Pac}
Peter Duren, Dmitry Khavinson, Harold~S. Shapiro, and Carl Sundberg,
  \emph{Contractive zero-divisors in {B}ergman spaces}, Pacific J. Math.
  \textbf{157} (1993), no.~1, 37--56. \MR{MR1197044 (94c:30048)}

\bibitem{D_Ap}
Peter Duren and Alexander Schuster, \emph{Bergman spaces}, Mathematical Surveys
  and Monographs, vol. 100, American Mathematical Society, Providence, RI,
  2004. \MR{MR2033762 (2005c:30053)}

\bibitem{tjf1}
Timothy Ferguson, \emph{Continuity of extremal elements in uniformly convex
  spaces}, Proc. Amer. Math. Soc. \textbf{137} (2009), no.~8, 2645--2653.

\bibitem{tjf2}
Timothy Ferguson, \emph{Extremal problems in {B}ergman spaces and an extension
  of {R}yabykh's theorem}, Illinois J. Math. \textbf{55} (2011), no.~2,
  555--573 (2012). \MR{3020696}

\bibitem{Hansbo}
J.~Hansbo, \emph{Reproducing kernels and contractive divisors in bergman
  spaces}, J. Math. Sci. (New York) \textbf{92} (1998), no.~1, 3657--3674.

\bibitem{Zhu_Ap}
Haakan Hedenmalm, Boris Korenblum, and Kehe Zhu, \emph{Theory of {B}ergman
  spaces}, Graduate Texts in Mathematics, vol. 199, Springer-Verlag, New York,
  2000. \MR{1758653 (2001c:46043)}

\bibitem{Hedenmalm_canonical_A2}
H{\aa}kan Hedenmalm, \emph{A factorization theorem for square area-integrable
  analytic functions}, J. Reine Angew. Math. \textbf{422} (1991), 45--68.
  \MR{MR1133317 (93c:30053)}

\bibitem{Khavinson_McCarthy_Shapiro}
Dmitry Khavinson, John~E. McCarthy, and Harold~S. Shapiro, \emph{Best
  approximation in the mean by analytic and harmonic functions}, Indiana Univ.
  Math. J. \textbf{49} (2000), no.~4, 1481--1513. \MR{MR1836538 (2002b:41023)}

\bibitem{Khavinson_Stessin}
Dmitry Khavinson and Michael Stessin, \emph{Certain linear extremal problems in
  {B}ergman spaces of analytic functions}, Indiana Univ. Math. J. \textbf{46}
  (1997), no.~3, 933--974. \MR{MR1488342 (99k:30080)}

\bibitem{MacGregor_Stessin}
T.~H. MacGregor and M.~I. Stessin, \emph{Weighted reproducing kernels in
  {B}ergman spaces}, Michigan Math. J. \textbf{41} (1994), no.~3, 523--533.
  \MR{MR1297706 (95k:46038)}

\bibitem{Olver:2010:NHMF}
F.~W.~J. Olver, D.~W. Lozier, R.~F. Boisvert, and C.~W. Clark (eds.),
  \emph{{NIST Handbook of Mathematical Functions}}, Cambridge University Press,
  New York, NY, 2010, Print companion to \cite{NIST:DLMF}.

\bibitem{Ryabykh}
V.~G. Ryabykh, \emph{Extremal problems for summable analytic functions},
  Sibirsk. Mat. Zh. \textbf{27} (1986), no.~3, 212--217, 226 ((in Russian)).
  \MR{MR853902 (87j:30058)}

\bibitem{Shapiro_Approx}
Harold~S. Shapiro, \emph{Topics in approximation theory}, Springer-Verlag,
  Berlin, 1971, With appendices by Jan Boman and Torbj\"orn Hedberg, Lecture
  Notes in Math., Vol. 187. \MR{MR0437981 (55 \#10902)}

\bibitem{Sundberg}
Carl Sundberg, \emph{Analytic continuability of {B}ergman inner functions},
  Michigan Math. J. \textbf{44} (1997), no.~2, 399--407. \MR{1460424
  (98h:46022)}

\bibitem{Dragan_point_eval}
Dragan Vukoti{\'c}, \emph{A sharp estimate for {$A^p_\alpha$} functions in
  {${\bf C}^n$}}, Proc. Amer. Math. Soc. \textbf{117} (1993), no.~3, 753--756.
  \MR{MR1120512 (93d:46042)}

\bibitem{Dragan}
\bysame, \emph{Linear extremal problems for {B}ergman spaces}, Exposition.
  Math. \textbf{14} (1996), no.~4, 313--352. \MR{MR1418027 (97m:46117)}

\bibitem{Zhu_Fock}
Kehe Zhu, \emph{Analysis on {F}ock spaces}, Graduate Texts in Mathematics, vol.
  263, Springer, New York, 2012. \MR{2934601}

\end{thebibliography}
\end{document}